\documentclass[12pt]{amsart}

\usepackage{amscd}
\usepackage{amsfonts}
\usepackage{amsmath}
\usepackage{amssymb}
\usepackage{amsthm}
\usepackage{amsxtra}
\usepackage{enumerate}
\usepackage{setspace}
\usepackage[offset=0pt,centering]{geometry}
\newtheorem{lemma}{Lemma}[section]
\newtheorem{theorem}[lemma]{Theorem}

\newtheorem{definition}[lemma]{Definition}
\newtheorem{example}[lemma]{Example}

\newtheorem{conjecture}[lemma]{Conjecture}
\newtheorem*{acknowledgement}{Acknowledgment}

\newcommand{\bn}{\textnormal{bn}}
\newcommand{\C}{\mathfrak{C}}

\newcommand{\D}{\mathfrak{D}}
\newcommand{\dn}{\textnormal{dn}}
\newcommand{\G}{\mathfrak{G}}
\newcommand{\Hh}{\mathfrak{H}}

\newcommand{\m}{\mathfrak{m}}

\newcommand{\R}{\mathfrak{R}}
\newcommand{\Ss}{\mathfrak{S}}
\newcommand{\Z}{\mathcal{Z}}

\begin{document}
\title{Decomposition Theory}
\author{Yi Zhang}
\dedicatory{Dedicated to my wife Huiqiong Deng}
\address{Department of Mathematics, Purdue University,
West Lafayette, IN 47907}
\email{zhan1291@purdue.edu}
\address{{\it https://sites.google.com/site/zhan1291/}}
\date{}

\begin{abstract}
We give a characterization of decomposition theory in linear algebra.
\end{abstract}
\maketitle

\section*{Introduction}

Since the introduction of abstract algebra, the study of decomposition of algebraic and geometric structures has been a central topic in mathematics. However, without the division operation, a general ring behaves far from a field, which makes decomposition theory fascinating yet intractable.

This paper introduces an elementary approach to this topic and initiates the study of decomposition number.

\section{Decomposition Number}

Throughout this paper, $R$ is a base ring and modules are left $R$-modules.

Let $M$ and $M_k$ $(k=1,\cdots,n)$ be $R$-modules. If there are $R$-morphisms $i_k:M_k\to M$ and $p_k:M\to M_k$ $(k=1,\cdots,n)$ such that 
\[p_ki_k=id_{M_k},\ \ p_ki_l=0 \ \ (k\neq l)\]
and 
\[\sum_k i_kp_k=id_M,\]
then $M_1\oplus \cdots\oplus M_n$ is a direct decomposition of $M$ in $R$-modules. If in this decomposition, no $M_k$ has nontrivial direct decomposition, then it is an indecomposable decomposition. 
If any two indecomposable decompositions of the module $M$ 
share the same indecomposable summands up to isomorphism and counting multiplicities, then we say the module $M$ 
satisfies the Krull-Schmidt condition.

\begin{definition}
If $M$ satisfies the Krull-Schmidt condition with the indecomposable decomposition $M=I_1\oplus \cdots\oplus I_n,$ then the decomposition number $\dn(M)$ is defined as $n.$ If $I$ is an indecomposable module and $M=I^{n_I}\oplus M'$ such that $I$ is not a direct summand of $M',$ then the decomposition number $\dn(M,I)$ relative to $I$ is defined as $n_I.$ 
\end{definition}

Let $M$ be a module over $R,$ and let $\G=\{g_i\}$ be a set of nonzero generators of $M.$ Define the associated free $R$-module $F$ as $\sum Re_i,$ where $\{e_i\}$ is a free basis. Define the relationship submodule of $F$ as $\{\sum r_ie_i\: | \: \sum r_ig_i=0\}.$ And define a relationship set $\R$ of $F$ as a set of generators of the relationship submodule. Each relationship set $\R$
defines an equivalence $\sim$ on the basis element $\{e_i\}$ as follows: 
\begin{itemize}
\item for each $i,$ $e_i\sim e_i,$ 
\item if $r_ie_i+r_je_j +\sum_{k\neq i,j}
r_ke_k\in \R$ where $r_i\neq 0$ and $r_j\neq 0,$ then $e_i\sim e_j,$
\item if $e_i\sim e_j$ and $e_j\sim e_k,$ then $e_i\sim e_k.$
\end{itemize}
This equivalence depends on the choice of the generators $\G$ as well as the relationship set $\R.$ Denote the number of equivalence classes in this equivalence by $n_{\G,\R}.$ Since a submodule of $M$ generated by all $g_i$'s whose corresponding $e_i$'s are in the same equivalence class is a direct summand of $M,$ the following criterion for indecomposability follows immediately.

\begin{theorem} \label{T:I}
A module $M$ is indecomposable if and only if there is only one equivalence class in $\{e_i\}$ for any choice of generators and relationship sets of $M.$ 
\end{theorem}

Since we could choose the generators of a module from its direct summands, we get a characterization of the decomposition number. 

\begin{theorem} \label{T:DN} 
If $M$ satisfies the Krull-Schmidt condition, then
$$\dn(M)=\max_{\G,\R} \{n_{\G,\R} \}.$$
\end{theorem}

More generally, we may define $\dn(M)$ as $\sup \{n_{\G,\R} \}$ for all $R$-modules, see Conjecture \ref{C:FD}. The relative decomposition number $\dn(M,I)$ can be studied similarly.

\section{Linear Algebra}

Let $R$ be a noetherian ring with identity such that finitely generated modules have unique minimal resolutions up to isomorphism, for example a noetherian local ring.

If $M$ is a finitely generated $R$-module, let $\G$ and $\R$ be minimal bases of the module $M$ and the relationship submodule in the corresponding free basis $\{e_i\},$ then $v=|\G|$ and $u=|\R|$ are independent of the minimal presentation of $M.$ We use a relationship matrix $A_{\G,\R}=(a_{ij})_{u\times v}$ to represent $\R,$ where each row $(a_{i1},\cdots, a_{iv})$ of $A_{\G,\R}$ corresponds to an element $\sum a_{ij}e_j$ in $\R.$

Suppose $\Ss$ is another minimal relationship set represented by a matrix $A_{\G,\Ss}.$ Since $(\R)=(\Ss),$ the rows of $A_{\G,\R}$ generate the rows of $A_{\G,\Ss}$ and vice versa. Therefore, there is an invertible matrix $P$ such that $A_{\G,\Ss}=P\cdot A_{\G,\R}.$ 

Suppose $\Hh$ is another minimal basis with corresponding free basis $\{f_i\}.$ Then there is an invertible transformation matrix $Q$ between the free bases such that $(e_i)=Q\cdot (f_i),$ and $A_{\G,\R}\cdot Q$ represents the relationship set in $\{f_i\}$ induced from $\R.$ 

Therefore, a relationship set $\R$ in $\{e_i\}$ is represented by a relationship matrix $A_{\G,\R}.$ And the relationship matrices of different choices of minimal bases of the module and the relationship submodules are $P\cdot A_{\G,\R}\cdot Q$ for invertible square matrices $P$ and $Q,$ which are equivalent to $A,$ or $P\cdot A_{\G,\R}\cdot Q\sim A$. 

In general, let $A$ be a $u\times v$-matrix over $R.$ We say $A$ has $t$ disjoint columns if for each $k$ such that $1\leq k\leq t,$ there are $n_k(>0)$ columns in $A$ such that their nonzero rows have entries zero in all other $v-n_k$ columns. We call the disjoint columns the blocks of $A,$ and call the maximal number of blocks the block number $\bn(A).$ If $A$ is equivalent to a matrix in $R$ with at least two blocks, then $A$ is blockable. Otherwise, $A$ is inblockable.

The blocks of $A_{\G,\R}$ correspond to the direct summands of $M,$ in particular the columns with entries 0 correspond to the free direct summands of $M.$
Therefore we have the following equivalent criterion of indecomposability as Theorem \ref{T:I}:

\begin{theorem}\label{T:C}
The module $M$ is indecomposable if and only if $A_{\G,\R}$ is inblockable for some minimal basis $\G$ and some relationship set $\R.$
\end{theorem}
\begin{proof}
Let $\Hh=\{h_i\}_{i=1}^{v}$ and $\Ss$ ($|\Ss|=u$) be minimal bases of the module $M$ and the relationship submodule.
If $\G=\{g_i\}$ is a minimal basis of $M$ and let $\R$ ($|\R|=u'\geq u$) be a relationship set. Then $A_{\Hh,\Ss}=P\cdot A_{\G,\R}\cdot Q$ for a $u'\times u$ transformation matrix $P$ from $\R$ to $\Ss$ and an invertible $v\times v$ transformation matrix $Q$ on the corresponding free bases of $\Hh$ and $\G.$ 
Since $\Ss$ is a minimal basis, there is a $u\times u'$ matrix $P'$ such that $P'\cdot P$ is the $u\times u$ identity matrix. Therefore, the matrix $P'\cdot P\cdot A_{\Hh,\Ss}\cdot Q$ is equivalent to $A_{\Hh,\Ss}.$ 
Hence $A_{\G,\R} =P\cdot A_{\Hh,\Ss}\cdot Q$ is inblockable if and only if $A_{\Hh,\Ss}$ is inblockable.
\end{proof}

Similarly, we have a description of the decomposition number as Theorem \ref{T:DN}:

\begin{theorem} 
If $\G$ is a minimal basis of $M$ and $\R$ is a relationship set, then 
$$\dn(M)=\max_{A\sim A_{\G,\R}} \{\bn(A) \}.$$
\end{theorem}

\section{Isomorphism}

Let $R$ be a noetherian ring with identity such that finitely generated modules have unique minimal resolutions up to isomorphism.

Define the category $\C$ of equivalence classes of finite dimensional matrices in $R$ as follows. The objects are finite dimensional matrices with the equivalence $\sim$ such that
\begin{itemize}
\item $P\cdot A\cdot Q \sim A$ for square invertible matrices $P$ and $Q,$
\item $(A, 0)^T\sim A^T,$ 
\item $(1)\sim 0$ where 0 is the empty matrix. 
\end{itemize}
If $[A_{u\times v}]$ and $[B_{s\times t}]$ are in $\C,$ then a morphism from $[A]$ to $[B]$ is an ordered pair of matrices $\{S_{u\times s},T_{v\times t}\}$ such that $A\cdot T=S\cdot B,$ under the equivalence compatible with the one on the objects.
The direct sum of $[A]$ and $[B]$ is $\left[ \begin{pmatrix} A & 0\\ 0 & B \end{pmatrix}\right].$ The category $\C$ is an abelian category.

If $[A]$ is in $\C$ such that $A=(a_{ij})_{u\times v}$ has no block $(1),$ then there is a finitely generated $R$-module $M_A=\oplus_i Re_i/(\R)$ where $\R=\{\sum_j a_{ij}e_j\}.$ The module $M_A$ has a relationship matrix $A.$ If $[\{S,T\}]$ is a morphism from $[A]$ to $[B]$ in $\C,$ then the matrix $T$ induces a transformation on the corresponding free bases of $M_A$ and $M_B,$ hence an $R$-morphism from $M_A$ to $M_B.$  

Let $\D$ denote the category of isomorphism classes of finitely generated $R$-modules.
If $[M]$ is in $\D,$ then there is a finite dimensional matrix $A_M$ which is a relationship matrix of $M.$ If $[N]$ is in $\D$ with minimal bases $\Hh$ and $\Ss,$ then an $R$-morphism from $M$ to $N$ is determined by the transformation matrix $T$ on the corresponding free bases of $\G$ and $\Hh,$ which also induces a transformation $S$ from $\R$ to $\Ss.$ The pair $\{S,T\}$ is a morphism from $A_M$ to $A_N.$

Therefore, we have the correspondence between decomposition theory and linear algebra as follows:

\begin{theorem} \label{T:CI}
The categories $\C$ and $\D$ are isomorphic.
\end{theorem}

\section{Example}

Theorem \ref{T:C} provides a construction of indecomposable modules, as demonstrated below.

\begin{example}\label{E:A}
If $R$ is a commutative noetherian local ring such that $\dim_R \textnormal{soc}R>1,$ then $R$ has infinitely many torsion-free indecomposable modules.
\end{example}
\begin{proof}
Let $n$ be a natural number, let $x$ and $y$ be two different socle elements in $R,$ and let $\Z$ be the module
$\Z=\bigoplus_{i=1}^{n+1}Re_i/(xe_i+ye_{i+1}\:|\: 1\leq i\leq n),$
where $\{e_i\}$ is a free basis. Then $\Z$ has a relationship matrix 
$$A=\begin{pmatrix}
x & y & 0 & 0 & \cdots \\
0 & x & y & 0 & \cdots \\
\cdots & \cdots & \cdots & \cdots & \cdots \\
\cdots & \cdots & 0 & x & y
\end{pmatrix}_{n\times (n+1)}.$$

Suppose $\Z$ is decomposable, then $A$ is blockable. So there are invertible $n\times n$ matrix $P$ and $(n+1)\times (n+1)$ matrix $Q$ such that $$P\cdot A\cdot Q=
\begin{pmatrix}
B & 0 \\
0 & C
\end{pmatrix},
$$ where $B$ and $C$ are blocks.
Since $\m\cdot x =\m\cdot y=0,$ we could regard the matrices $P$ and $Q$ in $k=R/\m.$ Without loss of generality, assume that $B$ is a $s\times t$ matrix of such that $s<t.$ Since $x$ and $y$ are linearly independent over $k,$ we may replace $x$ and $y$ by variables $X$ and $Y.$ Then over the field $k(X,Y),$ there is a nonzero vector $v$ such
that $B\cdot v=0.$ Hence
$$
A\cdot Q\cdot \begin{pmatrix}v\\0\end{pmatrix}=P^{-1}\begin{pmatrix}
B & 0 \\
0 & C
\end{pmatrix}
\begin{pmatrix} v\\0\end{pmatrix}=0,
$$
and $Q\cdot \begin{pmatrix} v\\0\end{pmatrix}$ is in the solution space of $A\cdot V=0$ over $k(X,Y),$ which is $k(X,Y)\cdot \left(Y^n, -XY^{n-1}, \cdots, (-X)^n\right)^T.$ However, $Q^{-1}\cdot \left(Y^n, -XY^{n-1}, \cdots, (-X)^n\right)^T$ does not have entries 0, which is a contradiction. 
\end{proof}

\section{Conjecture}

The author would like to propose the following conjecture regarding the functorial behavior of the decomposition number.

\begin{conjecture} \label{C:FD}
Let $f:R\text{-mod}\to R\text{-mod}$ be an additive functor, and let $M$ be an $R$-module such that $\dn(f^n(M))<\infty$ for all $n,$  
then (hopefully without additional conditions)
\begin{enumerate}
 \item[(a)] 
 $\lim_n \log_2 \dn (f^n(M))/n$ exists, 
 \item[(b)] 
 $\sum_n \dn (f^n(M))\cdot t^n$ is a rational function.
\end{enumerate}
The same conclusion holds for the relative decomposition numbers.
\end{conjecture}

The $F$-signature \cite{HL02,kT12} is a special case of (a).

\begin{acknowledgement} 
The author would like to thank Craig Huneke and Gennady Lyubeznik for their continued support, as well as Yongqiang Chen and Uli Walther for several discussions.
\end{acknowledgement}


\begin{thebibliography}{99}

\bibitem{HL02}
C. Huneke, G. J. Leuschke, Two theorems about maximal Cohen-Macaulay modules. Math. Ann. 324 (2002), no. 2, 391-404.

\bibitem{kT12}
K. Tucker, $F$-signature exists. To appear in Inventiones Mathematicae, arXiv:1103.4173, 2012.


\end{thebibliography}
\end{document}